\documentclass[12pt]{amsart}
\usepackage[utf8]{inputenc}

\makeatletter
\def\@settitle{\begin{center}%
		\baselineskip14\p@\relax
		\normalfont\Large
		\@title
	\end{center}%
}
\makeatletter
\def\@settitle{\begin{center}%
		\baselineskip14\p@\relax
		\normalfont\Large
		\@title
	\end{center}%
}
\makeatother

\usepackage{graphicx,epsfig}

\usepackage{bbm}
\usepackage[title]{appendix}
\usepackage{esvect}
\usepackage{tikz-cd}
\usepackage{fancyhdr}
\usepackage{amsfonts}
\usepackage{amsmath}
\usepackage{amssymb}
\usepackage{mathabx}
\usepackage{amsthm}
\usepackage{tikz-cd}
\usepackage{parskip}
\usepackage{times} 

\makeatletter
\newcommand\xleftrightarrow[2][]{%
  \ext@arrow 9999{\longleftrightarrowfill@}{#1}{#2}}
\newcommand\longleftrightarrowfill@{%
  \arrowfill@\leftarrow\relbar\rightarrow}
\makeatother

\usepackage{mathrsfs}
\usepackage{scalerel,amssymb}

\usepackage{amsmath,amssymb,latexsym, amsfonts, amscd, amsthm, xy}
\input{xy}
\xyoption{all}

\input{xy}
\xyoption{all}

\makeindex 

=650pt \headheight = 8pt \headsep = 10pt

\usepackage{fancyhdr}
\usepackage{amsfonts}
\usepackage{amsmath}
\usepackage{amssymb}

\usepackage{amsthm}
\usepackage{tikz}
\usepackage{tikz-cd}
\usepackage{parskip}
\usepackage{times} 

\usepackage[
bibstyle=alphabetic,
citestyle=alphabetic,
hyperref=true,
backref=false]{biblatex}
\usepackage{hyperref}
\hypersetup{pdftoolbar=true, pdftitle={GKform}, pdffitwindow=true, colorlinks=true, citecolor=green, filecolor=black, linkcolor=blue, urlcolor=blue, hypertexnames=false}
\usepackage{cleveref}
\makeindex 

\renewcommand{\rm}{\mathrm}

\newcommand{\Z}{\mathbb{Z}}
\newcommand{\Q}{\mathbb{Q}}
\newcommand{\R}{\mathbb{R}}

\newcommand{\F}{\mathbb{F}}

\newcommand{\PP}{\mathbb{P}}

\newcommand{\PGL}{\mathrm{PGL}}

\newcommand{\GL}{\mathrm{GL}}
\newcommand{\St}{\mathrm{St}}
\newcommand{\OS}{\mathrm{OS}}
\newcommand{\Tor}{\mathrm{Tor}}
\newcommand{\Ext}{\mathrm{Ext}}

\newcommand{\Hom}{\mathrm{Hom}}

\DeclareMathSymbol{\mathinvertedexclamationmark}{\mathclose}{operators}{'074}
\DeclareMathSymbol{\mathexclamationmark}{\mathclose}{operators}{'041}

\makeatletter
\newcommand{\raisedmathinvertedexclamationmark}{%
  \mathclose{\mathpalette\raised@mathinvertedexclamationmark\relax}%
}
\newcommand{\raised@mathinvertedexclamationmark}[2]{%
  \raisebox{\depth}{$\m@th#1\mathinvertedexclamationmark$}%
}
\begingroup\lccode`~=`! \lowercase{\endgroup
  \def~}{\@ifnextchar`{\raisedmathinvertedexclamationmark\@gobble}{\mathexclamationmark}}
\mathcode`!="8000
\makeatother

=650pt \headheight = 8pt \headsep = 10pt

\newtheorem{thm}{Theorem}[section]
\newtheorem{prop}[thm]{Proposition}
\newtheorem{lem}[thm]{Lemma}
\newtheorem{cor}[thm]{Corollary}

\newtheorem{defn}[thm]{Definition}
\theoremstyle{definition}
\newtheorem{rem}[thm]{Remark}

\numberwithin{equation}{section}
\title[Presenting restricted Steinberg modules]{Presenting restricted Steinberg modules of general linear groups}
\author{Peter Xu}

\usepackage{setspace}
\begin{document}

        \makeatletter
\@setabstract
\makeatother
\maketitle

    \begin{abstract}
        We give generators and presentations for various local restrictions of Steinberg modules over fields and relate them to \emph{partial} Borel--Serre compactifications of locally symmetric spaces in the case of number fields, extending the existing theory of partial modular symbols for $\GL_2(\Q)$. Along the way, we clarify the relationship between ``circuit'' and ``Bykovskii''-type presentations for such modules. In an appendix, we relate the existence of such presentations to Koszulity properties of the Steinberg $\mathrm{VB}$-algebra.
    \end{abstract}

    \tableofcontents

    \section{Introduction}

    The Steinberg representation $\St_G$ attached to an algebraic group $G$ over a field $F$ is a ubiquitous construction in number theory and representation theory, naturally appearing in a wide array of cohomological and geometric phenomena. Correspondingly, it has various definitions, the most common being as the top-dimensional reduced homology of the spherical \emph{Tits building} $\widetilde{H}_{r-1}(T_G)$, where $r$ is the $F$-rank of $G$ modulo the center, and $T_G$ is a simplicial complex built from the poset of parabolic subgroups of $G$. The celebrated \emph{Solomon--Tits theorem} says that this homology group totally captures the homotopy type of this building, in the sense that $T_G$ is homotopy equivalent to a bouquet of $(r-1)$-spheres.
    
    For general linear groups, we may write $\St_n(F)$ for $\St_{\GL_n/F}$ (and similarly $T_n(F)$ for the building); for a fixed free $F$-module\footnote{Here, we use language nonspecific to fields as a few instances of Steinberg modules for more general rings will occur.} $M$, we may also write $\St(M)$. In this case, parabolics are the same as flags of subspaces, and one can give a concrete and simple description of $T_n(F)$ as the simplicial complex whose $k$-simplices are flags
    \begin{equation} \label{eq:tits}
    0\subsetneq V_1 \subsetneq \ldots \subsetneq V_{k} \subsetneq F^n
    \end{equation}
    with the faces corresponding to forgetting intermediate subspaces; one sees this is indeed $(n-2)$-dimensional, one less than the $F$-rank of $\PGL_n$. 
    
    Steinberg modules can often be presented very explicitly, especially in the general linear setting. One important presentation was described by Ash--Rudolph \cite{AR}, which they called ``modular symbols,'' as for $\GL_2/\Q$ it gives classical modular symbols parameterizing the geometry of geodesics between cusps. This presentation has as generators by symbols $[\ell_1,\ldots, \ell_n]$ for independent lines $\ell_1,\ldots, \ell_n\subset F^n$, antisymmetric in the lines and subject to the additional ``cyclic'' relation
        \[
        \sum_{i=1}^{n+1} (-1)^i[\ell_1,\ldots, \hat{\ell}_i,\ldots, \ell_{n+1}]=0.
        \]
    for any tuple of $n+1$ lines; here, any non-independent symbols are set to zero. In terms of the geometry of the Tits building of $\GL_n$, these generators come from a certain set of canonical cycles called ``apartments'' coming from simple linear algebraic data, corresponding to a direct sum decomposition of $V$ into lines in the general linear setting. 
    
    In some cases, one has smaller presentations: for instance, given an integral subring $R\subset F$, one may ask if \emph{unimodular} apartment classes, whose lines span a fixed $R$-lattice in $F^n$, already suffice. (More generally, even for non-domains $R$, we may speak of \emph{$R$-unimodular} tuples of lines, whose generators form $R$-bases.)
    
    In a different direction of simplification, one may ask if only \emph{some} of the cyclic relations above are necessary: for instance, the ``quadratic relations'' wherein there is a linear relation between some triple $\ell_i,\ell_j,\ell_k$. (It turns out that this follows readily from the cyclic presentation in general linear setting, if one does not ask for unimodularity.)

    These simplifications, when they hold, can lead to much more tractable computations of modular symbols at finite level; when $n=2$ and $R=\Z$, the unimodular simplification was essentially used by Manin \cite{Ma} to give his eponymous ``Manin symbol'' formalism for modular symbols for $\GL_2/\Q$ at finite level.

    \subsection{Arithmetic group cohomology and modular symbols}
    
    Our primary number-theoretic interest in these objects is linked to arithmetic group (co)homology: Borel--Serre proved \cite{BS} that if $G$ is reductive with anisotropic center over $\Q$, $\St_{G}$ is a \emph{Bieri--Eckmann dualizing module} for arithmetic subgroups of $G(\Q)$. Indeed, let $X_G$ be the archimedean symmetric space attached to $G$; it admits a left action of $G(\Q)$ with isolated finite stabilizers. 
    \begin{thm} \label{thm:bs}
        Let $S$ be a set of finite places of $\Q$, and $\Gamma\le G(\Q)$ an arithmetic subgroup. If one takes coefficients in which all orders of torsion elements of $\Gamma$ are invertible, then we have canonical isomorphisms of homology groups
        \begin{equation} \label{eq:bs}
            H_\bullet(\Gamma,\St_G) \cong H_\bullet^{BM}(\Gamma\backslash X_G).
        \end{equation}
    \end{thm}
    One can give a proof\footnote{This is not \emph{precisely} the proof given originally by Borel--Serre, which exploited the theory of Bieri--Eckmann duality groups more directly and so does not obviously generalize. In essence, however, the underlying geometric picture is not very different.} of \eqref{eq:bs} by identifying the boundary of the \emph{Borel--Serre compactification} $\overline{X}_{G}$ of the archimedean symmetric space as having the homotopy type of $T_G$, and then descending (by a spectral sequence computation) to the locally symmetric space for $\Gamma$ whose cohomology yields group cohomology, up to the torsion mentioned. We will recall the construction of this compactification and associated results in Section \ref{section:bs}. Thanks to \eqref{eq:bs}, the representation theory of $\St_G$ thus becomes an indispensable tool in understanding the cohomology of $\Gamma$, and especially in proving vanishing results.
    
    We generalize the above proof to \emph{partial} Borel--Serre compactifications of $X_G$ in the case $G=\mathrm{Res}^F_\Q\,\GL_n$,\footnote{In fact, our methods work for general reductive groups $G$, but it is only in the general linear case that we have some interesting partial compactifications in mind, and applications for the resulting submodules of $\St_G$.} in terms of certain ``restricted'' Steinberg modules. By ``partial,'' we mean that we will only include boundary components corresponding to \emph{some} parabolic subgroups, coming from linear algebraic conditions at some finite set of nonarchimedean places of $F$; consequently, we obtain comparison results analogous to \eqref{eq:bs} for arithmetic subgroups of $G$ fixing these local conditions:

    \begin{thm} \label{thm:a}
        Let $F$ be a number field whose ring of integers $\mathcal{O}_F$ has proper prime and coprime ideals $I$ and $J$ of residue characteristic not equal to $2$, and let $S$ and $T$ be subspaces of $(\mathcal{O}_F/I)^n$ and $\mathcal{O}_F/J)^n$ respectively. Let $\Gamma$ be a torsion-free arithmetic subgroup of $\GL_n(F)$ stabilizing $S$ and $T$. Let ${}^S\mathfrak{C}$, respectively ${}_T\mathfrak{C}$, be the set of parabolic subgroups of $\GL_n(F)$ whose stabilized flag consists of subspaces not contained in $S$ modulo $I$, respectively does not contain $T$ modulo $J$. Then the relative homology of the partially compactified locally symmetric space (as defined in Section \ref{section:bs}), for the cusps $\mathfrak{C}={}^S\mathfrak{C}$ or ${}_T\mathfrak{C}$
        \[
        H_{n-1}(\Gamma \backslash \overline{X}_{\GL_n(F)}^{ \mathfrak{C}}, \Gamma \backslash \partial \overline{X}_{\GL_n(F)}^{\mathfrak{C}}, \Z)
        \]
        is canonically identified with the $\Gamma$-coinvariants of a certain $\Gamma$-module ${}^S\St_n(F)$, respectively ${}_T\St_n(F)$, with the following generators and relations:
        \begin{itemize}
            \item ${}^S\St_n(F)$ is generated by symbols $[\ell_1,\ldots, \ell_n]$ indexed by tuples of lines in $F^n$ none of whose reductions modulo $I$ are contained in $S$, subject to the relations in Definition \ref{def:byk}.
            \item ${}_T\St_n(F)$ is generated by symbols $[\ell_1,\ldots, \ell_n]$ indexed by tuples of lines in $F^n$ none of whose spanned hyperplanes modulo $J$ are contained in $T$, subject to the relations in Definition \ref{def:byk}.
        \end{itemize}
        Moreover, in either case, restricting to the symbols $[\ell_1,\ldots, \ell_n]$ which are $R$-unimodular, for any localization $R$ of $\mathcal{O}_F$ which is a PID (if $F$ has a real place) or is $R$ is semi-local (otherwise), still suffices to generate the module in question.
    \end{thm}
    
    Note that in the previously-considered case $n=2$, $G$ has $\Q$-rank $1$, the Tits building is zero-dimensional, and the boundary components of $X_G$ do not intersect, simplifying the calculations. The proof of the identification of restricted Steinberg modules with partially compactified homology is in Section \ref{section:bs}, and follows standard methods dating back to Borel and Serre. This formalism applies to any subset of cusps/parabolic subgroups such that the associated subcomplex of the building of the group is connective enough, not just to the particular types of arithmetic restrictions we are able to obtain generators/presentations for in the above theorem.
    
    The bulk of the paper is devoted to the generators/presentations results for Steinberg modules which are the other part of Theorem \ref{thm:a}. The main work here consists of bootstrapping existing topological and matroid-theoretic results about unrestricted Steinberg modules to obtain our desired conclusions: along the way, we prove connectivity results for the associated restricted Tits buildings (Propositions \ref{prop:CM1} and \ref{prop:CM2}), construct Whitney homology-type resolutions (after Bj\"orner \cite{Bj}) for the restricted modules (Proposition \ref{prop:whitney}), and relate the matroidal circuit/Orlik--Solomon type presentations to Bykovskii-type presentations (Proposition \ref{prop:bykcir}). 

    Our main generation/presentation results encapsulated in the theorem above are the unimodular apartment class generation (Proposition \ref{prop:unimod}), which we are able to bootstrap from the simplicial arguments of \cite{Sca} and \cite{CFP}, and the Bykovskii presentation for the restricted modules which we deduce by matroid-theoretic methods (Proposition \ref{prop:Sbyk}).

    \subsection{A word of motivation}

    The arithmetic motivation from the putative extension classes associated to the boundary of locally symmetric spaces is more moral than practical: the locally symmetric spaces for general linear groups are rarely Shimura varieties, making it difficult to endow the extension classes with any Hodge or Galois structure. Our more immediate motivation is that we make extensive use of restricted Steinberg modules (and Orlik--Solomon-type resolutions thereof) in several papers on explicit  ``Eisenstein'' cocycles: to date, namely the articles \cite{X2} and (jointly) \cite{RX2} over $\Q$, and (jointly) \cite{SX} over function fields. A major purpose of this article is therefore to give a unified and precise treatment to the corresponding definitions and presentations. This goal is especially necessitated by the pair of forthcoming articles \cite{X3} \cite{X4}, in which we will again use restricted Steinberg modules (over $\Q$, or imaginary quadratic fields) to construct explicit partial modular symbols valued in motivic cohomology groups attached to elliptic schemes, corresponding to multiple \emph{elliptic} polylogarithms. In particular, one of our major interests in these motivic-valued modular symbols is to give constructions of new cases of the Sharifi conjectures, describing delicate torsion information in the $K$-theory of fields.

    All of the above-mentioned articles develop the idea, implicitly already present in \cite{SV} and also pursued in \cite[\S4]{CRR}, that Steinberg modules parameterize distinguished ``polylogarithm elements'' of arithmetic interest inside families of $n$-fold powers of one-dimensional groups. Our partial Steinberg modules equipped with local restrictions arise naturally in these settings because the simplest elements forming the building blocks, corresponding to $n=1$, need auxiliary local smoothing to be defined (e.g. theta functions on elliptic curves), and/or need to satisfy certain linear-algebraic conditions with respect to the \emph{torsion} of the algebraic groups to define their specializations.

    \subsection{Acknowledgements}

    I would like to thank Romyar Sharifi for motivating me to understand these issues, as well as Avner Ash and Andy Putman for helpfully answering my emails.

    \section{Simplicialities}

    We assemble some basic background on simplicial complexes in this section.

    \subsection{Posets}
    Let $P$ be any poset. Its (categorical) nerve $\mathcal{N}(P)$ is the simplicial set whose collection of $k$-simplices consists of chains
    \begin{equation} \label{eq:chain}
        x_0\le \ldots \le x_{k},
    \end{equation}
    with face maps given by forgetting each of the $x_\bullet$ in turn, and degeneracy maps given by doubling each of the $x_\bullet$ in turn (cf. \cite{R1}). 
    
    The ordinary geometric realization $\Delta(P):=|\mathcal{N}(P)|$ (cf. \cite{Sca}) can be identified with the classical \emph{order complex} attached to $P$: this is the simplicial complex whose vertices are objects of $P$, and $n$-simplices are given by strict chains
    \[
        x_0< \ldots < x_k.
    \]
    Given elements $x\le y$ in the poset $P$, we may form the \emph{interval} $(x,y)$, which is the subposet of elements strictly between $x$ and $y$. When $P$ is clear from the context, we may write $H_\bullet(x,y)$ for the homology of intervals.
    
    We will also consider some simplicial complexes associated to covers: let $X$ be a set, and $\mathcal{C}:=\{U_i\}_{i\in I}$ a family of subsets of $X$. The nerve $\mathcal{N}(\mathcal{C})$ is the simplicial complex whose vertices are the basic opens, with simplices corresponding to nonempty intersections. One may relate this nerve to poset nerves as follows:
    \begin{defn}
        The associated poset $P(\mathcal{C})$ has as objects finite subsets $J\subset I$ for which 
        \[
        \bigcup_{i\in J} U_i \ne \emptyset
        \]
        ordered by inclusion. 
    \end{defn}
    The nerve of the poset $\mathcal{N}(P(\mathcal{C}))$ is then the barycentric subdivision of $\mathcal{N}(\mathcal{C})$.

    \subsection{Connectivity conditions}

    Let $S$ be a simplicial complex of dimension $d$. We say that $S$ is \emph{$d$-spherical} if it has the homotopy type of a bouquet of $d$-spheres. This is equivalent to asking for it to be $(d-1)$-connected (i.e. have vanishing homotopy groups in degrees $\le d-1$), and also to asking for it to be simply connected with homology concentrated in top degree (by the Hurewicz and Whitehead theorems). If one drops the simply-connected requirement, we say instead that it is \emph{homologically} $d$-spherical.

    Let $\sigma \subset S$ be any simplex of dimension $s\le d$. Let the \emph{link} of $\sigma$, $\rm{Lk}(\sigma)\subset S$, be the subcomplex consisting of the union of all simplices $\tau\subset S$ disjoint from $\sigma$, such that the union of the vertices of $\tau$ and of $\sigma$ is a simplex of $S$. Then $\rm{Lk}(\sigma)\subset S$ is a subcomplex of dimension at most $d-s-1$. This motivates the following stronger version of sphericality, which amounts to adding that the dimension locally looks as expected everywhere:
    \begin{defn}
        The $d$-dimensional simplicial complex $S$ is (homologically) \emph{Cohen--Macaulay} if every simplex of dimension $s$ has link which is (homologically) $(d-s-1)$-spherical.
    \end{defn}
    Note that by the Hurewicz theorem, the only difference between the homological and homotopical notion is simple connectivity. In this article, we will always work with the weaker \emph{homologically} Cohen--Macaulay condition, and so we drop the ``homological'' going forward. Most, or perhaps all, of the complexes we will consider this for are surely actually homotopy Cohen--Macaulay, but the extra simple connectivity statement does not seem relevant to any of our applications, so we take the simpler route.
    
    If $S=|\mathcal{N}(P)|$ for a poset $P$, and $\sigma$ corresponds to a chain $x_0<\ldots < x_k$, then we can interpret $\rm{Lk}(S)$ as having simplices corresponding to chains which pass through none of the objects $x_0,\ldots, x_k$, but whose objects are comparable to all of them.

    \section{Steinberg modules}

    \subsection{Tits buildings}

    Let $R$ be a quotient or localization of a Dedekind domain,\footnote{We do not work in the maximum generality possible since we are mostly interested in some number theoretic applications.} and let $\mathrm{Fr}_n(R)$ be the poset of proper nonzero $R$-free summands of $R^n$. We define the Tits building $T_n(R)$ to be its order complex; this is the simplicial complex whose $i$-simplices are given by flags
    \[ \{0\} \ne F_0 \subset \ldots \subset F_{p}\subset R^n\]
    with the natural face relations. Note that when $R$ is a PID, one can check from this definition that the natural inclusion $T_n(R)\to T_n(\mathrm{Frac}(R))$ is an isomorphism.
    
    Bj\"{o}rner \cite[Theorem 4.1]{Bj2} proved the following strengthening of the classical Solomon--Tits theorem when $R$ is a field, which was generalized by Scalamandre \cite{Sca} to all the rings we are considering (and others):
    \begin{thm}
        The building $T_n(R)$ is Cohen--Macaulay for any quotient or localization of a Dedekind domain.
    \end{thm}
    In particular, $T_n(R)$ is $(n-3)$-connected, and its reduced top homology we define to be the Steinberg module for $R$:
    \[
    \St(R):= \tilde{H}_{n-2}(T_n(R)).
    \]
    Again, this is the same as the Steinberg module for its fraction field when $R$ is a PID. We may distinguish a set of cycles which really do depend on the ring $R$, and not simply its fraction field: these are the $R$-integral apartment classes 
    \[
    [\ell_1, \ldots, \ell_n]:= \sum_{\sigma\in S_n} (-1)^{\rm{sgn}\,\sigma} [\langle \ell_{\sigma(1)}\rangle \subset \ldots \subset \langle \ell_{\sigma(1)},\ldots, \ell_{\sigma(n-1)}\rangle].
    \]
    Here, $\ell_1,\ldots, \ell_n$ are rank-one summands (which we will informally refer to as ``lines'') which together are \emph{unimodular}, i.e. span $R^n$, as mentioned in the introduction. In fact, even for $k<n$ lines $\ell_1,\ldots, \ell_k$, one may ask that they span a rank-$k$ free summand and define a corresponding cycle 
    \[
        [\ell_1, \ldots, \ell_k] \in C_{k-2}(T_n(R))
    \]
    which is the image of an apartment class coming from an inclusion $T_k(R)\to T_n(R)$, though these classes are visibly boundaries (as is forced by the connectivity of $T_n(R)$).

    It is classical that the apartment classes generate when $R$ is a field; moreover, if $R$ is a PID, $T_n(R)=T_n(F)$. In this setting, Maazen proved $\St_n(F)$ is generated by $R$-integral apartment classes for Euclidean $R$ \cite{Maa} and Church--Putman--Farb \cite{CFP} the same for $R$ with a non-complex place, with at least two places which are either archimedean or inverted in $R$. Beyond PIDs, \cite[Theorem B]{Sca} showed that $\St_n(R)$ is generated by apartment classes for any Dedekind domain with a real place, as well for a larger class of rings including local and finite Artinian ones. 

    The Tits building always has an involution we denote by
    \[
        \delta: T_n(R) \to T_n(R), [S]\mapsto [S^\perp]
    \]
    which on $0$-simplices sends a free summand of $R^n$ to its orthogonal complement under the standard inner product on $R^n$, and reverses the order of flags. Under the natural left action of $g\in \GL_n(R)$ on the Tits building, we have $g\circ \delta = \delta \circ (g^T)^{-1}$. One may check on the level of the top-dimensional apartment classes that this involution induces
    \begin{equation} \label{eq:apartdelta}
        [\ell_1,\ldots, \ell_n] \mapsto [\ell_1^\vee,\ldots, \ell_n^\vee]    
    \end{equation}
    where $\ell_i^\vee$ is the orthogonal complement of the span of $\ell_1,\ldots, \hat{\ell}_i, \ldots, \ell_n$. Below top dimension, $\delta$ does not send apartment cycles to apartment cycles; relatedly, one sees that $\delta$ is not equivariant for inclusions of lower-dimensional Tits buildings.

    \subsubsection{Restricted complexes}
    We will study some restricted versions of the Tits building, parameterized by a subspace $S$ which is a proper nonzero free summand of some quotient $(R/I)^n$ for some proper ideal $I$ of a subring of $R$. We then define the ``upwards'' and ``downwards'' $S$-avoiding posets
    \begin{itemize}
        \item ${}^S\mathrm{Fr}_n(R)$, the poset of proper nonzero free summands of $R^n$ which are not contained in $S$ modulo $I$, and
        \item ${}_S\mathrm{Fr}_n(R)$, the poset of proper nonzero free summands of $R^n$ which do not contain $S$ modulo $I$.
    \end{itemize}
    We then define the restricted Tits complexes ${}^ST_n(R)$ and ${}_ST_n(R)$ to be the respective order complexes, still of dimension $n-2$; these are subcomplexes of $T_n(R)$. We let ${}^S\St_n(R)$ and ${}_S\St_n(R)$ to be the corresponding reduced top homologies. 
    
    If $R$ is a domain with field of fractions $F$, we may also similarly define ${}^S\mathrm{Fr}_n(F)$, ${}_S\mathrm{Fr}_n(F)$, ${}^ST_n(F)$, ${}_ST_n(F)$ ${}^S\St_n(F)$, and ${}_S\St_n(F)$, in the same manner, since subspaces of $F^n$ yield free and cofree summands of $R^n$ which may be reduced modulo $I\subset R$. If $R$ is a PID, there is no distinction between these definitions.\footnote{More generally, as in \cite[\S4.4]{Sca}, one may interpolate between these two constructions by picking subsets of the Picard group of isomorphism classes of projective $R$-modules.} 

    We call these \emph{restricted} Steinberg modules. For any tuple of lines $\ell_1,\ldots, \ell_n$ for which no proper subset spans a subspace contained in $S$ modulo $I$, respectively no line is contained in $S$ modulo $I$, we have an apartment class $[\ell_1,\ldots, \ell_n]$ in ${}^S\St_n(R)$, respectively ${}^S\St_n(R)$, (or also the $F$-variants when $R$ is a domain) though it is not yet clear these generate. We may also consider variants ${}^S_T\St_n(R), {}^S_TT_n(R)$ (and their $F$-analogues), etc. by imposing \emph{both} types of condition, modulo a pair of relatively prime ideals of $R$.

    We have the following duality between the upwards and downwards restrictions which is immediate from definitions:
    \begin{prop} \label{prop:delta}
    The map $\delta$ restricts to give a pair of inverse isomorphisms between ${}^{S}_TT_n(R)$ and ${}^{T^\perp}_{S^\perp}T_n(R)$ (or their $F$-analogues), and thus an inverse-transpose twisted $\GL_n(R)$-equivariant identification between the restricted Steinberg modules given as before by the map \eqref{eq:apartdelta}.
    \end{prop}

    \subsection{Matroids and geometric lattices}

    We briefly recall some basic notions from matroid theory; our reference for this material is \cite{Bj}. A \emph{matroid} on a ground set $E$ is a subset of the power set $2^E$ which roughly captures the notion of ``linearly independent sets''. Considering $2^E$ as an abstract simplicial complex (whose $p$-simplices are the $(p+1)$-subsets, with evident face maps), one may define a matroid to precisely be a simplicial subcomplex $\mathcal{M}\subset 2^E$ satisfying the \emph{exchange axiom}: if a simplex $S_1$ has more vertices than $S_2$, then there exists a vertex $v\in S_1\setminus S_2$ such that $\{v\} \cup S_2$ is a simplex of $\mathcal{M}$. 

    The \emph{rank} $r(\mathcal{M})$ of the matroid is its one more than its dimension as a simplicial complex, i.e. the maximal size of an independent set; we may simply write $r$ at times when there is no ambiguity about the matroid in question. More generally, one may similarly define the rank of a subset of $E$ (for $\mathcal{M}$) as the maximal independent subset. For each $1\le d \le r(\mathcal{M})$, one then may define the rank-$d$ \emph{flats} as the maximal sets of rank $d$: these form a lattice $L(\mathcal{M})$ under intersection (and in fact can be used to give an alternate axiomatization of matroids). A \emph{circuit} is a minimal \emph{dependent} set, i.e. such that removing any element results in an independent set. 
    
    The lattice $L(\mathcal{M})$ has a minimal element corresponding to the empty set of rank zero; the \emph{atoms} $A(\mathcal{M})$ are the minimal nonzero elements of $L(\mathcal{M})$; it also has a unique maximal flat corresponding to the entire set $E$. Define $T(\mathcal{M})$ to be the order complex of the sublattice of nonzero nonmaximal flats. Bj\"{o}rner proved that both $\mathcal{M}$ itself and $T(\mathcal{M})$ are Cohen--Macaulay simplicial complexes of dimension $r-1$, respectively $r-2$, in particular having reduced homology concentrated in top degree.
    
    The \emph{Orlik--Solomon algebra} associated to $\mathcal{M}$ was first introduced in \cite{OS}, though in the specific setting of hyperplane arrangements rather than general matroids. Let $\Lambda(E)_\bullet$ be the free augmented graded commutative algebra over $\Z$ generated in degree $1$ by the set $E$. It has a cdga structure given by extending the augmentation in degree $1$ using the Leibniz rule; we then may define the Orlik--Solomon ideal $J_{\mathcal{M}}$ to be the cdga ideal generated by 
    \[
        \partial([e_1]\wedge \ldots \wedge [e_k])
    \]
    for each circuit $\{e_1,\ldots, e_k\}$ of $\mathcal{M}$; one may check that $\OS(\mathcal{M})_\bullet:=\Lambda(E)_\bullet/J_{\mathcal{M}}$ is an acyclic cdga concentrated in degrees $0 \le \bullet \le r(\mathcal{M})$. 
    
    The identification in the following theorem was originally proven by Orlik--Solomon \cite{OS} for hyperplane arrangements, though in somewhat different non-matroidal language. The complete statement in the generality below is due to Bj\"{o}rner \cite{Bj}:
    \begin{prop} \label{thm:OS and order}
         We have an ``apartment map''
        \[
        \mathrm{OS}(\mathcal{M})_{r} \to \tilde{H}_{r-2}(T(\mathcal{M})),\;\;\; [e_1]\wedge \ldots \wedge [e_r] \mapsto \sum_{\sigma \in S_r} (-1)^{\mathrm{sgn}\, \sigma}[ \langle e_{\sigma(1)}\rangle \subset \ldots \subset \langle e_{\sigma(1)},\ldots, e_{\sigma(r-1)}\rangle]
        \]
        which is an isomorphism.
    \end{prop}
    The classical examples of matroids are arrangements of nonzero vectors in $F^n$ for a field $F$, with the simplices corresponding to linearly independent sets; the corresponding lattice of flats is the lattice of subspaces spanned by subsets of these vectors. (One may also formulate things in terms of the dual hyperplanes and their intersections, which will arise later.) We will primarily be concerned with (minor variants of) the rank-$n$ \emph{projective matroid}, consisting of ground set $\mathbb{P}^{n-1}(F)$ and the natural independence condition.
    \begin{rem}
        Some of the literature we cite pertaining to matroid/Orlik--Solomon algebras is formulated for only \emph{finite} matroids, hyperplane arrangements, etc. while the full projective matroid will be infinite unless $F$ is finite. Generally, most of the constructions (matroid homology, the Orlik--Solomon algebra, etc.) are covariant in the matroid, allowing us to simply take filtered direct limits everywhere and preserve any homological calculations. For a few things, we will need to be more careful; we will mention when this occurs.
    \end{rem}

    We will call this matroid $\mathcal{P}_n(F)$; it is of rank $n-1$, and its lattice of flats is precisely the poset $\mathrm{Fr}_n(F)$ (minus zero), and the atoms correspond to the lines of the ground set; the corresponding order complex is then the Tits building $T_n(F)$.

    \subsection{Connectivity and Whitney homology of restricted buildings}
    
    Let $R$ be a domain with fraction field $F$, $S$ and $T$ proper nonzero subspaces modulo relatively prime ideals $I$ and $J$ of $R$. Define ${}^{S}\mathcal{P}_n(F)$ to be the full submatroid of the projective matroid on the subset of the ground set $\PP^{n-1}(F)$ whose reduction modulo $I$ is not contained in $S$. Then its corresponding order complex is precisely ${}^{S}T_n(R)$. Bj\"{o}rner \cite{Bj} proved that any matroid lattice has Cohen--Macaulay order complex; further, Proposition \ref{thm:OS and order} implies apartment classes always generate the top homology. The duality isomorphism between upwards- and downwards-restricted Tits complexes therefore immediately implies:

    \begin{prop} \label{prop:CM1}
    ${}^{S}T_n(F)$ and ${}_TT_n(F)$ are Cohen--Macaulay of dimension $n-2$; their top homologies ${}^S\St_n(F)$ and ${}_T\St_n(F)$ are generated by $F$-apartment classes.
    \end{prop}
    
    The matroidally-defined ${}^S\St_n(F)$ has an Orlik--Solomon algebra, but not ${}_T\St_n(F)$ or ${}^S_T\St_n(F)$. However, they do have \emph{resolutions} by Orlik--Solomon-like complexes without natural algebra structures, which may be useful in applications. These are inspired by the identification of the Orlik--Solomon algebra of a matroid with the \emph{Whitney homology} of a graded poset \cite[\S7.9]{Bj}, which we explain briefly in our particular setting; the spectral sequence argument below is taken from \cite{Bac}. 
    
    For a subposet $P$ of $\mathcal{P}_n(F)$, we have an ascending filtration by rank
    \[
        \emptyset = F_0P \subset F_1P \subset \ldots \subset F_{n-1}P = P
    \]
    whose $i$th part consisting of subspaces with rank at most $i$. The homological spectral sequence for the associated stratification on $\Delta(P)$ has $E^1$ page given by
    \[
    E^1_{p,q} = H_{p}(F_{q}P, F_{q-1}P) \cong \bigoplus_{\mathrm{rk}\, M = q} \tilde{H}_{p-1}(\hat{0}, M)
    \]
    where recall this latter notation denotes the homology of the order complex of an interval; the lower bound $\hat{0}$ simply means that we do not impose a lower bound on the interval. For $P={}^S_T\mathcal{P}_n(F)$, each interval of the form $(\hat{0},M)$ is the $S\cap M/I$-restricted Tits complex for the free module $M$, ${}^{S\cap M/I}\mathcal{P}(M)$, because since $M$ does not contain $T$ in the reduction modulo $J$, neither does any submodule of $M$. These intervals therefore have order complexes which are Cohen--Macaulay, with top reduced homology a restricted Steinberg module for $M$. The $E^1$-page of the spectral sequence therefore reduces to the following so-called ``Whitney homology'' exact sequences:

    \begin{prop} \label{prop:whitney}
        We have exact sequences
        \[
        0\to {}_T\St_n(F) \to \bigoplus_{\substack{M/J\not\supset T\\ \mathrm{rk}\, M = n-1}} \St(M) \to \bigoplus_{\substack{M/J\not\supset T\\ \mathrm{rk}\, M = n-2}} \St(M) \to \ldots \to \bigoplus_{\substack{M/J\not\supset T\\ \mathrm{rk}\, M = 1}} \St(M) \to \Z \to 0
        \]
        and 
        \[
        0\to {}_T^S\St_n(F) \to \bigoplus_{\substack{M/J\not\supset T\\ \mathrm{rk}\, M = n-1}} {}^{S\cap M}\St(M) \to \bigoplus_{\substack{M/J\not\supset T\\ \mathrm{rk}\, M = n-2}} {}^{S\cap M}\St(M) \to \ldots \to \bigoplus_{\substack{M/J\not\supset T\\ \mathrm{rk}\, M = 1}} {}^{S\cap M}\St(M) \to \Z \to 0
        \]
        where the maps are the Whitney differentials defined previously. Here, we write $S\cap M$ for $S\cap M/I$ for short, and when this is zero we simply consider ${}^{S\cap M}\St(M)$ to be the unrestricted Steinberg module.
    \end{prop}

    \begin{prop} \label{prop:CM2}
        The doubly-restricted complex ${}^S_T T_n(F)$ is Cohen--Macaulay.
    \end{prop}
    \begin{proof}
        The spectral sequence computation in the proof of the preceding proposition shows that ${}^S_T T_n(F)$ is homologically spherical of dimension $n-2$. Moreover, one has the identification, general for order complexes of posets, that the link of a simplex corresponding to a flag $M_1 \subset M_2 \subset \ldots \subset M_k$ is the join of order complexes of intervals 
        \[
        \Delta(\hat{0},M_1) \ast \Delta(M_1,M_2) \ast \ldots \ast \Delta(M_k, \hat{1}).
        \]
        As we saw previously, the first order complex is a singly-restricted one; the middle $k-1$ order complexes are actually unrestricted order complexes attached to $M_i/M_{i-1}$ (since $M_{i-1}/I$ is not contained in $S$ and $M_i/J$ does not contain $T$), and the last one is a singly-restricted order complex associated to $F^n/M_k$, consisting of subspaces whose reduction mod $J$ does not contain the image of $T+M_k/J$ in $(R/J)^n/(M_k/J)$, since $M_k$ is not contained in $S$. These factors are al Cohen--Macaulay, so we are done.
    \end{proof}

    We also record the following locally-integral apartment generation result:

    \begin{prop} \label{prop:unimod}
        Assume $R$ is a PID which is either semi--local or a Dedekind domain with a real place, or two non-archimedean places. Then ${}^S\St_n(F)$ and ${}_T\St_n(F)$ are generated by $R$-integral apartment classes.
    \end{prop}
    \begin{proof}
        By duality, it suffices to prove the result for the upward restriction $S$. Let ${}^SB_n(R)$ be the \emph{complex of restricted partial bases}, the simplicial complex of subsets of lines in $R^n$ which form a basis for a free/cofree summand, and such that no line is contained in $S$ modulo $I$. By the argument of \cite[\S5.1]{Sca}, since we already know that ${}^ST_n(R)$ is Cohen--Macaulay, it suffices to show that ${}^SB_n(R)$ is Cohen--Macaulay. This is a full subcomplex (on the vertices corresponding to \emph{good} lines not contained in $S$) of the non-restricted analogue $B:=B_n(R)$ (defined identically but without the $S$ condition) proven Cohen--Macaulay in \cite[Theorem 3.11]{Sca} (noting that semi-local rings satisfy the stability condition there denoted $SR_2$). 
        
        One only has to do a little work to bootstrap from the non-restricted analogue to our desired result. We have the Hatcher--Vogtmann style deletion lemma (stated in the form we need in \cite[Proposition 2.5]{GRW}) saying that if $Z$ is a simplicial complex and $W \subseteq Z$ is a full subcomplex such that $\mathrm{lk}_Z(\beta) \cap W$ is $(d - \dim\beta - 1)$-connected for every simplex $\beta$ of $Z$ with no vertex in $W$, then $Z$ being $d$-connected implies $W$ is $d$-connected. We claim that $\mathrm{lk}_B(\tau) \cap {}^SB_n(R)$ is $(n - \dim\tau - 3)$-connected for every simplex $\tau$ of $B$: if $\tau$ is top-dimensional, there is nothing to check. Proceeding by downwards induction, we apply the deletion lemma with $Z = \mathrm{lk}_B(\tau)$, $W = \mathrm{lk}_B(\tau) \cap {}^SB_n(R)$, and $d = n - \dim\tau - 3$. As $B$ is Cohen--Macaulay of dimension $n-1$, the link $Z$ is Cohen--Macaulay of dimension $n - \dim\tau - 2$, hence $d$-connected. For a bad simplex $\beta$ of $Z$ we have $\mathrm{lk}_Z(\beta) \cap W = \mathrm{lk}_B(\tau * \beta) \cap {}^SB_n(R)$, which is $(d - \dim\beta - 1)$-connected by induction, since $\dim(\tau * \beta) = \dim\tau + \dim\beta + 1$. So $W$ is $d$-connected, and the claim follows.

        It remains to show that ${}^SB_n(R)$ is pure of dimension $n-1$: indeed, if $\{v_1, \dots, v_k\}$ is a simplex, we may extend it to a basis $v_1, \dots, v_n$ of $R^n$, and replace each bad $v_j$, $j > k$, by $v_j + v_1$. Hence ${}^SB_n(R)$ is $(n-1)$-dimensional with each $k$-simplex link pure of dimension $n - k - 2$, and it is $(n-2)$-connected with $(n-k-3)$-connected links. Hence it is $(n-1)$-spherical with $(n-k-2)$-spherical links.
    \end{proof}

    \subsection{Presentations} \label{section:prezi}

    As briefly mentioned in the introduction, there are two common closely-related forms of presentations for Steinberg modules one finds in the literature. We will distinguish the first presentation, originally written down by Ash--Rudolph \cite{AR}, by calling it \emph{circuit-type}. As mentioned before, it was originally written down by Ash--Rudolph without unimodularity stipulations; however, we will impose them here. For us, a \emph{circuit-type presentation} for $\St(R)$ has as generators symbols $[\ell_1,\ldots, \ell_n]$ indexed by $R$-unimodular ordered tuples of lines, with the relations:
    \begin{enumerate}
        \item (antisymmetry) $[\ell_{\sigma(1)},\ldots, \ell_{\sigma(n)}] = (-1)^{\rm{sgn}\,\sigma} [\ell_1,\ldots, \ell_n]$ for any $\sigma \in S_n$,
        \item (``circuit'' relations) for any tuple of $n+1$ lines $\ell_1,\ldots, \ell_{n+1}$ such that any independent sub-$n$-tuple is unimodular,
        \begin{equation} \label{eq:circuit}
        \sum_{i=1}^{n+1} (-1)^i[\ell_1,\ldots, \hat{\ell}_i,\ldots, \ell_{n+1}]=0.
        \end{equation}
        where any symbol with non-independent lines is taken to be zero.
    \end{enumerate}
    
    For $R$ a field (and thus without unimodularity conditions), as mentioned, these generators and relations originate in Ash--Rudolph. However, they did not actually claim that this gives a \emph{presentation}; the exhaustiveness of the relations was made explicit later in \cite{LS} \cite{KaS}. A similar result proven almost contemporaneously by Orlik--Solomon \cite{OS}; see next subsection.

    Bykovskii \cite{By} refined arguments of Ash--Rudolph to give a $\Z$-unimodular apartment presentation with a smaller set of relations. We refer to the following as a ``Bykovskii-type'' presentation of $\St_n(R)$: as generators, we take symbols $[v_1,\ldots, v_n]$ for $R$-unimodular tuples of primitive vectors in $R^n$, subject to the relations:
    \begin{enumerate} \label{def:byk}
        \item (antisymmetry) $[v_{\sigma(1)},\ldots, v_{\sigma(n)}] = (-1)^{\rm{sgn}\,\sigma} [v_1,\ldots, v_n]$ for any $\sigma \in S_n$,
        \item (homogeneity) for any unit $u\in R^\times$, $[uv_1,\ldots, v_n] = [v_1,\ldots, v_n]$, 
        \item (``Bykovskii'' relation)
        \begin{equation} \label{eq:byk}
        [v_1,\ldots, v_n]-[v_1+v_2,v_2,\ldots, v_n] + [v_1+v_2, v_1, v_3,\ldots, v_n] =0.
        \end{equation}
    \end{enumerate}
     Bykovskii proved this for $R=\Z$ originally by refining Ash--Rudolph's work \cite{By}; a new topological proof was given later by Church--Putman \cite{CP}, and generalized to $\Z[i]$ and $\Z[\omega]$ by Kupers-Miller--Patzt--Wilson \cite{KMPW} (who also showed its failure for some Euclidean domains).

    \begin{rem}
        Proposition \ref{thm:OS and order} implies apartment generation and the circuit-type presentation of $\St_n(F)$. As previously remarked, this presentation could therefore already be deduced quite formally from Orlik--Solomon's work \cite{OS}, though they did not write things down in precisely this language. (In particular, the flag algebra $\mathscr{B}$ under shuffle product of loc. cit. is the Whitney homology algebra of \cite[\S7]{Bj} constructed from order complexes of chains in $L(\mathcal{M})$, which we discuss below.)
    \end{rem}

    One immediately sees that if one uses the homogeneity relation to identify primitive vectors with lines in $R^n$, that the Bykovskii relation is a special case of the circuit relation: it is the ``maximally degenerate'' case where there is a triple of dependent vectors, or equivalently where all $n$ by $n$ minors fail to be independent besides three. In fact, conversely, the Bykovskii relation implies the circuit relations, over any $R$:
    \begin{prop} \label{prop:bykcir}
        The $R$-Bykovskii-type and $R$-circuit-type presentations are equivalent, under the natural map of generating symbols
        \[
        [v_1,\ldots, v_n]\mapsto [Rv_1,\ldots, Rv_n].
        \]
    \end{prop}
    \begin{proof}
        From the preceding discussion, we need only show that the Bykovskii-type relations generate all circuit relations: suppose we have an $(n+1)$-tuple of lines $\ell_1,\ldots, \ell_{n+1}$ such that every independent sub-$n$-tuple is unimodular over $R$; pick arbitrarily primitive vectors $v_1,\ldots, v_{n+1}$ generating these lines. Without loss of generality, suppose that the subtuple $v_1,\ldots, v_n$ is independent; then this tuple is also unimodular by assumption. We therefore conclude there is a unique relation of the form
        \[
        v_{n+1}=r_1v_1 + \ldots + r_{n}v_n
        \]
        for $r_1,\ldots, r_n\in R$. In fact, each $r_i$ must be either zero or a unit: for example, if $r_1$ were nonzero and not a unit, this would make $v_2,\ldots, v_{n+1}$ an independent $n$-tuple, and therefore unimodular. But computing the determinant in the unimodular basis given by $v_1,\ldots, v_n$, yields precisely $r_1$, which is therefore a unit.

        Replacing $v_i$ by $r_iv_i$ for $1\le i \le n$ whenever $r_i\ne 0$ (and thus $\in R^\times$) using homogeneity, and permuting if necessary to move all the vectors with nonzero $r_i$ up front, it therefore suffices to show that the circuit relations hold for the tuple $v_1+v_2+\ldots + v_k, v_1,v_2,\ldots, v_n$ for any $2\le k \le n$. 
        
        We do this by strong induction on $k$, which we call the ``arity'' of the relation: for the base case, note that $k=2$ is precisely the Bykovskii relation. Then general $k>2$, the circuit relation for $v_1+\ldots + v_k, v_1,\ldots, v_n$ is the sum of the circuit relations for the $(k-1)$-arity relations coming from the tuples
        \[
            (v_1+v_2)+v_3+\ldots +v_k, v_1+v_2,v_2, v_3, \ldots, v_k, v_{k+1}, \ldots, v_n
            \]
        with sign $+1$, and 
        \[
            (v_1+v_2)+v_3+\ldots +v_k, v_1+v_2, v_1, v_3, \ldots, v_k, v_{k+1}, \ldots, v_n
        \]
        with sign $-1$, with the alternating sum of the $2$-arity relations coming from the tuples
        \[
            (v_1+v_2)+v_3+\ldots +v_k, v_1+v_2, v_1, v_2, \ldots,\hat{v}_i \ldots, v_k, v_{k+1}, \ldots, v_n
        \]
        for $3\le i \le k$, and the $2$-arity relation for the tuple
        \[
        v_1+v_2, v_1,v_2,\ldots, v_n, 
        \]
        which are all of the form assumed in the inductive hypothesis.
    \end{proof}
    In particular, the validity of Bykovskii-type presentations over fields follows immediately from the (known) circuit-type presentations: the only extra difficulties come from integrality. From the classical matroid theory, we can obtain $F$-Bykovskii or circuit-type presentations for the singly-restricted Steinberg modules ${}^{S}\St_n(F)$ and ${}_{T}\St_n(F)$, under fairly mild assumptions:

    \begin{prop} \label{prop:Sbyk}
        For any proper nonzero subspace $S$ modulo $I$ as before such that $R/I$ is not characteristic $2$, or pair of nonzero proper subspaces $S,T$ modulo relatively primes ideals $I,J$ also with residue characteristic different from $2$, the restricted Steinberg modules ${}^{S}\St_n(F)$, ${}_{T}\St_n(F)$ have $F$-Bykovskii-type presentations. Explicitly, in the various cases, these presentations are:
        \begin{itemize}
            \item with generators given by independent tuples $[v_1,\ldots, v_n]$ of vectors in $F$ none of which span a line whose reduction is contained in $S$, which are antisymmetric, homogeneous, and satisfy the Bykovskii relation \eqref{eq:byk} for every tuple $v_1,\ldots, v_n$ for which $v_1+v_2$ is not contained in $S$; 
            \item with generators $[v_1,\ldots, v_n]$ such that no proper subtuple spans a subspace containing $T$ in the reduction, and satisfy \eqref{eq:byk} for any $(n+1)$-tuple for which no independent subtuple of size $\le n-1$ spans a subspace containing $S$ mod $I$.
        \end{itemize}
    \end{prop}
    \begin{proof}
        We begin with ${}^S\St_n(F)$. Proposition \ref{thm:OS and order} implies that ${}^S\St_n(F) \cong H_{n-2}({}^ST_n(F))$ has a circuit-type presentation over $F$: i.e., for antisymmetric generators 
        $[\ell_1,\ldots, \ell_n]$ on tuples of lines not contained in $S$ mod $I$, relations given by vanishing for $F$-dependent tuples and the circuit relation \eqref{eq:circuit} when all lines are not contained in $S$ mod $I$. 
        
        The proof of Proposition \ref{prop:bykcir} can be used to prove that also in this setting, this circuit-type presentation implies the Bykowskii-type presentation in the proposition statement: the reduction to proving a circuit relation to $v_1+\ldots + v_k,v_1,\ldots, v_n$ (for each of these not contained in $S$) is the same, and the intermediate relations used in the inductive step are always valid so long as we choose an ordering of the vectors so that $v_1+v_2$ is not contained in $S$: when $\mathrm{char}\, R/I>2$, this is always possible by reordering, since if every pairwise sum is contained in $S$, then $v_1+\ldots + v_k$ is also contained in $S$.
        
        Interpreting our antisymmetric symbols as formal wedge products, we claim that these circuit relations additively generate the relations in $\OS({}^S\mathcal{P}_n(F))_n$, meaning the circuit presentation over $F$ for ${}^S\St_n(F)$ would follow from Proposition \ref{prop:bykcir}. Indeed, suppose we have a circuit relation coming from a tuple of lines $\ell_0,\ldots, \ell_n$ avoiding $S$ modulo $I$; this is a dependent set, so contains a circuit of ${}^S\mathcal{P}_n(F)$ which we can assume after reordering is $\ell_0,\ldots, \ell_k$ for some $k\le n$. Then the corresponding circuit relation \eqref{eq:circuit} is the wedge product of the relation
        \[
        \partial ([\ell_0] \wedge \ldots \wedge [\ell_k])
        \]
        in degree $k$ of the Orlik--Solomon ideal with $[\ell_{k+1}]\wedge \ldots \wedge [\ell_n]$. Conversely, all the top-degree relations of the Orlik--Solomon algebra are (by definition) generated by wedge products of monomials with boundaries of circuits, and thus are ``circuit type relations'' as we have defined them.

        To obtain the Bykovskii-type presentation over $F$ for ${}_T\St_n(F)$, we recall the duality map
        \[
        \delta: {}^{T^\perp}\St_n(F) \xrightarrow{\sim} {}_T\St_n(F)
        \]
        which can be used to transfer over presentations. Note first that the generating apartment symbols $[\ell_1,\ldots, \ell_n]$ for lines not contained in $T^\perp$ mod $I$ are precisely bijected with symbols such that no span of a proper subtuple is contained in $T$ mod $I$. Furthermore, suppose we have a Bykowskii relation 
        \[
            [v_1,\ldots, v_n] - [v_1+v_2,v_2,\ldots, v_n] + [v_1+v_2,v_1,v_3\ldots, v_n],
        \]
        and let us write $w_i$, $1\le i\le n$ for a generator of the dual basis so that $\langle v_i, w_j\rangle =\delta_{i=j}$ for $1\le i,j\le n$; then $w_1,\ldots, w_n$ is also unimodular since $v_1,\ldots, v_n$ is. Then one computes that the $\delta$-dual of the Bykovskii relation is precisely
        \[
            [w_1,\ldots, w_n] - [w_1,w_2-w_1,w_3\ldots, w_n] + [w_2,w_1-w_2,w_3\ldots, w_n]
        \]
        which, after applying homogeneity and antisymmetry, is also a Bykovskii relation. Relations of the former form for which none of $v_1+v_2,v_1,v_2,\ldots, v_n$ is contained in $T^\perp$ mod $I$ are bijected under in this correspondence with relations of the latter form for which no proper subtuple of $w_1,\ldots, w_n$ spans a subspace containing $T$ mod $I$, completing the proof.
        \end{proof}

    Note that we have a circuit-type presentation for ${}^S\St_n(F)$ independent of the characteristic of $R/I$; the extra assumption is only needed to show this implies a Bykovskii-type presentation. This is needed for the duality argument for ${}_S\St_n(F)$ and the remainder of the proof, since the $\delta$-dual of a circuit presentation is not a circuit presentation unless it is of the ``quadratic'' form of Bykovskii; it may involve up to $\binom{n+1}{2}$ different lines in general. 
    
    \begin{rem}
        The restriction on the characteristic is necessary: for example, consider the matroid ${}^S\mathcal{P}_3(\F_2)$ where $S=\langle e_1+e_2, e_1+e_3\rangle$ for $e_1,e_2,e_3$ a basis. There are no circuits of length $3$ at all, and so no quadratic relations (but nevertheless obvious cubic ones corresponding to the circuit $\{e_1,e_2,e_3,e_1+e_2+e_3\}$). However, cubic circuit relations cannot substitute the quadratic Bykovskii-type relations in our argument, as the dual of a cubic circuit relation is no longer a circuit relation at all.
    \end{rem}

    Our matroid-theoretic methods do not appear to generalize to give unimodular presentations, nor to the doubly-restricted Steinberg modules.

    \section{Partial modular symbols} 

    \subsection{Locally symmetric spaces and the Borel--Serre construction} \label{section:bs}
    We now let $G$ be a reductive group over $\Q$ of $\Q$-rank $r$, $\mathfrak{P}$, $\mathfrak{B}$, and $\mathfrak{M}$ its parabolic, Borel, and maximal parabolic subgroups as before, and recall briefly its associated archimedean symmetric space $X_G$ and its Borel--Serre compactification $\overline{X}_{G}$,\footnote{Usually, this is decorated with a superscript $BS$, but since we consider no other compactifications in this article, we do not bother with this.} a manifold with corners. All of the below material is taken from \cite{BS} originally, though we follow more closely the treatment in \cite[III.5]{BJ}.
 
    We take the covering of $\mathfrak{B}$ given by
    \[
    \mathcal{C}_G:=\{ U_{M}\}_{M\in \mathfrak{M}},\; U_M:=\{B\in \mathfrak{B}: M\supset B\}\subset \mathfrak{B}.
    \]
    Then the Tits building $T_G$ is the simplicial complex given by the ordinary realization of the nerve $|\mathcal{N}(\mathcal{C}_G)|$. When $G$ is the Weil restriction of $\GL_n(F)$ to $\Q$, $T_G$ coincides with our earlier-defined building $T_n(F)$; note that the $F$-rank of a group over $F$ equals the $\Q$-rank of its Weil restriction. 

    Let 
    \[
    X_G:= G(\R) / K_\infty
    \]
    where $K_\infty \subset G(\infty)$ is the maximal compact; this is a smooth manifold which is homogenous for the left action of $G(\R)$, on which $G(\Z)$ acts virtually freely and properly. For each $P\in \mathfrak{P}$ with Levi decomposition $P=M_PN_P$ for $M_P=A_PM_P^0$ the Levi subgroup (with $A_P$ the connected component of the identity in the split central torus of $M_P$) and $N_P$ the unipotent radical, we have an associated horospherical decomposition 
    \[
    X_G \cong N_P \times A_P \times X_P
    \]
    where $X_P=M_P^0/\text{maximal compact}$ is the symmetric space for $M_P$ has a $P$-equivariant geodesic action of $A_P$ \cite[(III.5.29)]{BJ}. Identifying $A_P \cong \R^r_+$ using its simple roots, for $r$ the $\Q$-rank of $P$; denote its closure in $\R^r$ by $\overline{A}_P$. Write
    \[
    X(P):= X_{G\infty} \times_{A_P} \overline{A_P}\cong N_P \times X_P \times \overline{A_P}
    \]
    for the \emph{corner} associated to $P$. Denote then $e(P)$, the \emph{Borel--Serre boundary component}, to be the preimage of the point corresponding to the origin in $\overline{A}_P$. We then have a canonical decomposition
    \[
    X(P) = X\cup \bigsqcup_{P\subseteq Q} e(Q),
    \]
    and we can define the Borel--Serre compactification $\overline{X}_{G}$ as the colimit of all these $X(P)$ along all parabolics $P$; note that as sets, 
    \[
    \overline{X}_{G} = X_G \cup \bigsqcup_{P\in \mathfrak{P}} e(P),
    \]
    and that in general, for any $P,Q\in \mathfrak{P}$, $P\subseteq Q$ iff $e(P) \subseteq \overline{e(Q)}$. We then have:
    \begin{thm}[Borel--Serre]
        The compactification $\overline{X}_{G}$ is homotopy equivalent to $X_G$. Further, the natural action of $G(\Q)$ on the manifold with corners $\overline{X}_{G}$ is still virtually free, and the action of $G(\Z)$ is proper. 
    \end{thm}
    In short, one can check that the action of $G(\Q)$ permutes the boundary components corresponding to various $P$ by conjugation, and the action of each $P$ on its boundary component is free thanks to the presence of the unipotent component $N_P$.

    The connection to the Tits building comes from the following result, which is \cite[(8.2.2)]{BS}:

    \begin{thm}
        Let $Y$ be any space, with a locally finite cover $\{Y_i\}_{i\in I}$ by closed subsets, with nerve $N$. Suppose that $|N|$ is a finite-dimensional simplicial complex, and each $Y_i$ is an \emph{absolute retract} as defined in \cite[\S8.1]{BS}. Then $Y$ and $|N|$ have the same homotopy type.
    \end{thm}

    It turns out that the covering of $\overline{X}_{G}$ by $\{\overline{X(P)}\}_{P\in \mathfrak{P}}$ has this property, and thus $\partial \overline{X}_{G}$ has the homotopy type of $T_G$. We will simply exploit this more generally, for partial compactifications corresponding to subcomplexes of the nerve:

    \begin{cor} \label{cor:bs}
        Let $\mathfrak{S}\subset \mathfrak{P}$ be any set of parabolics which is a ``downward closed'' sub-poset under containment, i.e. such that if $P\in \mathfrak{S}$ and $Q\subseteq P$, then $Q\in\mathfrak{S}$; let $C_{\mathfrak{S}}\subset T_G$ the subcomplex coming from the nerve of the subcover of $\mathfrak{B}$. Then the partial compactification
        \[
        \overline{X}^{\mathfrak{S}}_{G} := X_G \cup \bigsqcup_{P\in \mathfrak{S}} e(P)
        \]
        has the homotopy type of $C_{\mathfrak{S}}$.
    \end{cor}

    \begin{proof}[Proof of Theorem \ref{thm:a}]
        As in the theorem statement, fix a pair of relatively prime ideals $I$ and $J$ of $\mathcal{O}_F$, proper subspaces $S$ and $T$ of $(\mathcal{O}_F/I)^n$ and $(\mathcal{O}_F/J)^n$, and let $\mathfrak{S}$ be the set of parabolics contained in a parabolic stabilizing $S$ and $T$, so that $C_{\mathfrak{S}}$ is identified with ${}^S_T T_n(F)$.
        
        Recall that $\Gamma$ is a torsion-free arithmetic subgroup of $G(\Q)$ fixing $S$ and $T$, $X_G$ is the associated symmetric space, and $\overline{X}^{\mathfrak{S}}_{G}$ the partial compactification associated to the data above; then $\Gamma$ acts freely on both $X_G$ and the boundary $\partial \overline{X}^{\mathfrak{S}}_{G}$. We have the Hochschild--Serre spectral sequence
        \[
        E^2_{p,q} = H_p(\Gamma, H_q(\overline{X}^{\mathfrak{S}}_{G},\partial\overline{X}^{\mathfrak{S}}_{G},\Z)) \Rightarrow H_{p+q}(\Gamma \backslash \overline{X}^{\mathfrak{S}}_{G},\Gamma \backslash \partial \overline{X}^{\mathfrak{S}}_G, \Z).
        \]
        From the long exact sequence of a pair and the contractibility of $X_G$, we deduce
        \[
        H_{n-1}(\Gamma \backslash \overline{X}^{\mathfrak{S}}_{G}, \Gamma \backslash \partial \overline{X}^{\mathfrak{S}}_{G}, \Z) \cong H_{n-2}(\overline{X}^{\mathfrak{S}}_{G},\Z),
        \]
        with every other degree of relative homology vanishing. Then Corollary \ref{cor:bs} implies that the resulting collapse of the spectral sequence on the $E_2$ page yields the identification
        \[
        {}^S_T\St(R)_\Gamma \cong H_{n-1}(\Gamma \backslash \overline{X}^{\mathfrak{S}}_{G}, \Gamma \backslash \partial \overline{X}^{\mathfrak{S}}_{G}, \Z).
        \]
        The rest of the theorem now follows from Proposition \ref{prop:unimod}.
    \end{proof}
    
    \begin{rem} \label{rem:functionfield}
    We expect that there should be similar presentations of partially compactified homology of the Bruhat--Tits building in the function field setting, though there are complications arising from the non-freeness of the group action. Kondo--Yasuda \cite{KY} give one approach without using compactifications, with unavoidable $p$-power torsion discrepancies, which we believe can likely be related to our presentations in the same formal kind of way, but we would prefer genuine partial compactifications - perhaps using the approach to compactifications described in \cite{FKS2}. This would permit a cleaner geometric formulation of the Eisenstein--Drinfeld modular symbols constructed in our joint article \cite{SX} in terms of homology of these buildings. 
    \end{rem}

    \appendix

    \section{Koszulity of VB objects and presentations}

    In this appendix, we briefly recall the Koszulity theory of Steinberg $\rm{VB}$-algebras from \cite[\S4]{MPW}, and relate Koszulity for such algebras to the construction of left resolutions (in particular, unimodular presentations) for Steinberg modules. Much of this material is in loc. cit. and \cite[\S3]{CRR}; a good reference for the classical Koszul duality of graded algebras and graded modules over these algebras (for comparison, as many concepts may be directly ported over) is \cite{PP}. 
    
    Fix a PID $R$, and let $\mathrm{VB}_R$ be the groupoid whose objects are the free $R$-modules of finite rank, and whose morphisms are their $R$-automorphisms. The category $\mathrm{VB}_R$ is symmetric monoidal under the direct sum of $R$-modules, with the natural braiding given by switching summands. As $R$ will be fixed in the rest of the section, we generally omit the subscript.
    
    A $\mathrm{VB}$-\emph{module} is a functor $\mathrm{VB} \to \mathrm{Ab}$; the monoidal structure on both sides endows these functors $\mathrm{Fun}(\mathrm{VB}, \mathrm{Ab})$ with a symmetric monoidal operation called \emph{Day convolution}. Explicitly, this is given by 
    \[
        (S \otimes T)(M) := \bigoplus_{M_1\oplus M_2 = M} S(M_1) \otimes T(M_2)
    \]
    for $\mathrm{VB}$-modules $S$ and $T$, and $M$ a free $R$-module; here, the direct sum is over internal direct sum decompositions. We take the braiding on this symmetric monoidal category given by 
    \[
      \sigma(M):(S\otimes T)(M) \to (T\otimes S)(M),\;\; S(M_1) \otimes T(M_2)\ni x\otimes y \mapsto (-1)^{\mathrm{rk}\,M_1\, \mathrm{rk}\,M_2}y\otimes x \in T(M_2) \otimes S(M_1).
    \]
    A $\mathrm{VB}$-\emph{algebra} is defined to be a monoid in $\mathrm{VB}$-modules, using the symmetric monoidal structure above. We always assume that the trivial $R$-module is assigned $\Z$; this is the analogue of an augmentation in the classical setting.
    
    Using the standard diagrams internal to the category of $\mathrm{VB}$-modules, one has also notions of $\mathrm{VB}$-modules over a $\mathrm{VB}$-algebra, $\mathrm{VB}$-ideals, $\mathrm{VB}$-coalgebras, etc. In the rest of this subsection we freely drop the ``$\mathrm{VB}$'' prefix as being understood. 

    \begin{defn}
    We define the \emph{free lineal} $\mathrm{VB}$-\emph{algebra} $\mathcal{F}$ (or $\mathcal{F}_R$ if $R$ is not clear in the context) generated in degree $1$ by 
        \[
        \mathcal{F}(M) := \bigoplus_{M=\ell_1\oplus \ldots \oplus \ell_d} \Z
        \]
    where the direct sum is over internal direct sum decompositions into lines, and the Day product is given by concatenation of internal direct sum decompositions. 
    \end{defn}

    More generally, one could define free tensor $\mathrm{VB}$-algebras for arbitrary assignments of modules to the lines in $\mathrm{VB}$, but we have no need for this here.

    \begin{defn}
    The \emph{Steinberg $VB$-algebra} $\St_R$ (or just $\St$ if $R$ is clear from context) is the assignment $M\mapsto \St(M)$ on free $R$-modules $M$, with the Day product $\St(V) \otimes \St(W) \to \St(V\oplus W)$ induced by pushforward by the smash maps
    \[
        \Sigma^2(T_\ell(R) \wedge T_{k}(R)) \to  T_{\ell+k}(R)
    \]
    defined simplicially by \cite{GKRW} via concatenation of flags; on homology, this is a braided-commutative product which in particular is given by concatenation on apartment class symbols
    \[
    [\ell_1,\ldots, \ell_a]\otimes [r_1, \ldots, r_b]\mapsto [\ell_1,\ldots, \ell_a, r_1,\ldots, r_b].
    \]
    When apartment classes generate all Steinberg modules over $R$, $\St$ is a quotient of the free lineal algebra $\mathcal{F}$ by an ideal $J$. If Steinberg modules over $R$ have Bykowskii-type presentations, then $\St$ is visibly \emph{quadratic}, meaning that $J$ is generated as a $\mathcal{F}$-module by its values on rank-$2$ $R$-modules.
    \end{defn}

    As with classical graded algebras, the derived tensor product $\Z\otimes^{\mathbb{L}}_A \Z$ may be computed by the reduced bar construction:
    \begin{equation}\label{eq:bar}
    B(A):=\cdots \to A_+^{\otimes 3} \xrightarrow{d_3} A_+^{\otimes 2} \xrightarrow{d_2} A_+ \xrightarrow{d_1} \Z \to 0,
    \end{equation}
    where $A_+$ is the augmentation ideal, i.e. the kernel of the augmentation map $A\to \Z$ which annihilates all the values of $A$ on positive-rank $R$-modules. The bar differential $d_k$ is the alternating sum $d_k = \sum_{i=1}^{k-1} (-1)^{i} m_i$, where $m_i:A_+^{\otimes k}\to A_+^{\otimes (k-1)}$ takes the Day product of the $i$-th and $(i+1)$-th tensor factors. It also carries a natural dg $\mathrm{VB}$-coalgebra structure with comultiplication given by deconcatenation $[a_1|\cdots|a_k]\mapsto \sum_{i=0}^{k} [a_1|\cdots|a_i]\otimes [a_{i+1}|\cdots|a_k]$, and its homology $\Tor^A_*(\Z,\Z)$ inherits a graded coalgebra structure. 
 
    The bar complex \eqref{eq:bar} is bigraded: a term $A_{M_1}\otimes \cdots \otimes A_{M_k}$ has homological degree $k$ and internal degree $\dim_R M_1 +\cdots+\dim_R M_k$; note that $B(A)_{p,q}=0$ whenever the homological degree $p$ is greater than the internal degree $q$. The phenomenon of \emph{Koszulity}, then, as in the classical setting, is the existence of a ``small'' formal model for $\Z\otimes^{\mathbb{L}}_A\Z$ via diagonal concentration:
    \begin{defn}\label{thm:koszul}
    A $\mathrm{VB}$-algebra $A$ is \emph{Koszul} if it satisfies $\Tor_i^A(\Z,\Z)_j = 0$ for $i\ne j$, in which case $\Z\otimes^{\mathbb{L}}_A\Z$ is quasi-isomorphic as a coalgebra to $A^{!`}$, defined as $M\mapsto \Tor_r^A(\Z,\Z)(M)$ for rank-$r$ $M$, via the obvious inclusion of the latter into the former. $A^{!`}$ is referred to as the \emph{Koszul dual $\mathrm{VB}$-coalgebra}.
    \end{defn}

    In particular, one computes immediately from the definition of the bar complex that if $A$ is a quotient of $\mathcal{F}$ by an ideal $J$, then on rank-$2$ modules $M$ we have $A^{!`}(M)=J(M)$. (More generally, there are descriptions in terms of the rank-$2$ values of $J$ of the Koszul dual in any rank, analogous to the descriptions \cite[Chapter I, Proposition 7.2]{PP} in the setting of classical graded algebras, but we do not need these.)
 
    $A^{!`}$ has the structure of a $\mathrm{VB}$-coalgebra, and has a \emph{cobar} construction 
    \[
    \Omega(A^{!`}) :=  \bigoplus_k (A^{!`}_+)^{\otimes k},
    \]
    computing its $\Ext$ groups in a dual story to the one given above, with a differential given by the alternating sum of the coproducts on $A^{!`}$ in each factor, and an algebra structure given by concatenation of tensors. As before, for the internal grading given by $R$-module rank, it is supported only on the locus where homological degree is at most internal degree. Its cohomology computes $R\Hom_{A^{!`}}(\Z,\Z)$, which inherits a graded algebra structure. We have the standard bar-cobar adjunction, as stated in \cite[Proposition 11]{CRR}:

    \begin{prop}
        For a Koszul algebra $A$, the map $\Omega(A^{!`})\to A$ induced by 
        \[
        \bigoplus_{\ell_1 \oplus \ldots \oplus \ell_r=M}A^{!`}(\ell_1) \otimes \ldots \otimes A^{!`}(\ell_r)\twoheadrightarrow A(M)
        \]
        on rank-$r$ $M$, factoring through the diagonally graded quotient of $\Omega(A^{!`})$, is a quasi-isomorphism of algebras, where $A$ is given the zero differential.
    \end{prop}

    Concretely, this yields the following left resolution of $A$ in terms of $A^{!`}$:

    \begin{cor}
        If $A$ is Koszul, then for any rank-$r$ $R$-module $M$, the piece of the cobar complex
        \begin{equation} \label{eq:resolution}
            A^{!`}(M) \to \ldots \to \bigoplus_{
        V_1 \oplus \ldots \oplus V_{r-1}=M}A^{!`}(V_1) \otimes \ldots \otimes A^{!`}(V_{r-1})\to \bigoplus_{
        V_1 \oplus \ldots \oplus V_r=M}A^{!`}(V_1) \otimes \ldots \otimes A^{!`}(V_r) \to A(M) \to 0
        \end{equation}
        is a resolution of $A(M)$.
    \end{cor}
    Note in particular that if $M$ is a quotient of $\mathcal{F}$, the last part of the above resolution is precisely the quadratic presentation of $M$: the last resolution term is identified with $\mathcal{F}(M)$, and each summand in the second to last term
    \[
    \bigoplus_{
        V_1 \oplus \ldots \oplus V_{r-1}=M}A^{!`}(V_1) \otimes \ldots \otimes A^{!`}(V_r)
    \]
    precisely contributes those rank-$2$ relations $J(V_i)$ coming from the unique $V_i$ which is rank $2$ instead of $1$; altogether, the image of the whole sum is the span of all relations induced from ``quadratic''/rank-$2$ relations.
    
    The above story breaks down cleanly by homological degree: if we refer to concentration of $\Tor_d^A(\Z,\Z)$ up to internal degree $d$ as ``$d$-Koszulity'', which is equivalent to $d$-connectivity of the map $\Omega(A^{!`})\to A$.
    \begin{cor} \label{cor:koszprezi}
        An algebra $A$ is $d$-Koszul if and only if the complex \eqref{eq:resolution} is exact at the last $d$ places (i.e. $1$-Koszulity is equivalent to surjectivity of the last map, etc.). 
    \end{cor}
    \begin{proof}
        For formal reasons, the looping-delooping map $\Omega(B(A))\to A$ is always a quasi-isomorphism; thus, the homology of the cone of $\Omega(A^{!`})\to A$ (which we need to show is $(d-1)$-connected) is identified with the homology of the map $\Omega(A^{!`})\to \Omega(B(A))$ induced by applying the cobar construction to the tautological map $A^{!`}\to B(A)$. By definition, $d$-Koszulity is equivalent to $d$-connectivity of this map, and the cobar construction lowers connectivity by at most $1$, from which the result follows.
    \end{proof}
    Thus, $1$-Koszulity is equivalent to generation in rank-$1$ (i.e. by apartment classes, as we called them in the Steinberg setting), $2$-Koszulity is equivalent to quadratic presentation, etc. just as in the classical setting.

    \begin{rem}
        Bykovskii-type presentations for Steinberg modules for $R$-modules imply Koszulity up to degree $2$, but the converse is not true: the Bykovskii relation is in particular a \emph{triangular} relation in a rank-$2$ $R$-module, involving only three lines. Rather, $2$-Koszulity is instead equivalent to an affirmative answer to \cite[Question 5.2]{KMPW}, which asks whether the rank-$2$ relations generate all Steinberg relations independently of whether $3$-cycles generate the rank-$2$ relations.
    \end{rem}

    \begin{rem}
        As in \cite[\S2]{CRR}, if we pass to rational coefficients, the Koszul duality of the commutative and Lie operads means that we have a smaller Chevalley--Eilenberg cobar resolution attached to the dual \emph{Lie} coalgebra of $A$ (arising by the Milnor--Moore theorem as the indecomposables of the Hopf algebra $A^{!`}$): concretely, for example, this means the degree-$2$ part of the cobar resolution will not need to package all the extraneous information about the anticommutativity of every pair of elements of $A$. We are interested in the integral structure, so do not permit ourselves this reduction.
    \end{rem}

    \begin{rem}
        The Koszulity of $\St$ for $R$ a field is closely related (and strictly weaker than) the Koszulity of the Orlik--Solomon algebras for the full projective matroid of every rank for $R$. The latter statement was proven by Yuzvinsky \cite{Yuz} by a general method of producing a so-called PBW basis for Orlik--Solomon algebras of certain matroids, using the upper-triangular basis of Steinberg modules adapted to a flag. (The original proof was stated only for rationalized Steinberg modules and for $R$ a finite field, but the PBW basis method works also for $\Z$-coefficients and infinite fields.)
    \end{rem}
    
    \printbibliography
\end{document}